\documentclass{llncs}
\usepackage{graphics,color,amssymb}
\newcommand{\mc}[1]{\ensuremath{\mathcal{#1}}}

\newcommand{\bbox}{\rule{.08in}{.08in}}
\newcommand{\class}[1]{\textsf{#1}}
\newcommand{\com}[1]{}

\newcounter{myclaim}
\begin{document}
\title{Segment representation of a subclass of co-planar graphs}
\author{Mathew C. Francis \and Jan Kratochv\'{i}l \and Tom\'{a}\v{s}
Vysko\v{c}il}
\institute{Department of Applied Mathematics, Charles University,\\
Malostransk\'{e} N\'{a}m\v{e}st\'{i} 25, 11800 Praha 1, Czech Republic.\\
\{francis,honza,whisky\}@kam.mff.cuni.cz}
\maketitle
\bibliographystyle{plain}
\begin{abstract}
A graph is said to be a segment graph if its vertices can be mapped to line
segments in the plane such that two vertices have an edge between them if and
only if their corresponding line segments intersect. Kratochv\'{i}l and
Kub\v{e}na \cite{KratoKube} asked the question of whether the complements of
planar graphs are segment graphs. We show here that the complements of all
partial 2-trees are segment graphs.
\end{abstract}

\section{Introduction}
Given a family of sets $\mathcal{F}$, a simple, undirected graph $G(V,E)$ is
said to be an ``intersection graph of sets from $\mathcal{F}$'' if there exists
a function $f:V(G)\rightarrow \mathcal{F}$ such that for $u,v\in V(G)$
such that $u\not=v$, $uv\in E(G)\Leftrightarrow f(u)\cap f(v)\not=\emptyset$. We
let $IG(\mathcal{F})=\{G~|~G\mbox{ is an intersection
graph of sets from }\mathcal{F}\}$. When $\mathcal{F}$ is a collection of
geometric objects, $IG(\mathcal{F})$ is said to be a class of ``geometric
intersection graphs''. Some well-known classes of geometric intersection
graphs are:
$$\class{INT}=IG(\{\mbox{all intervals on the real line}\})\mbox{ (``Interval
graphs'')}$$
$$\class{STRING}=IG(\{\mbox{all simple curves in the plane}\})\mbox{ (``String
graphs'')}$$
$$\class{CONV}=IG(\{\mbox{all convex arc-connected regions in the plane}\})$$
$$\class{SEG}=IG(\{\mbox{all straight line segments in the plane}\})\mbox{
(``Segment
graphs'')}$$

Clearly, $\class{SEG}\subseteq \class{CONV}\subseteq \class{STRING}$. The last
inclusion follows from the fact that string graphs are exactly the intersection
graphs of arc-connected regions in the plane.
It was shown in \cite{KratoKube} that the complement of every planar graph is
in $\class{CONV}$. In the same paper, the authors pose the question of whether
the complement of every planar graph is in $\class{SEG}$. A positive answer to
this question would imply that the \class{MAXCLIQUE} problem for segment graphs
is NP-complete, thus resolving a long standing open problem raised by
Kratochv\'{i}l and Ne\v{s}et\v{r}il in 1990 \cite{KratoNese}. It is
worth noting that the question of whether every planar graph is in \class{SEG},
known as Scheinerman's conjecture, was resolved by Chalopin and Gon\c{c}alves
\cite{Chalopin} who showed that every planar graph is indeed the intersection
graph of line segments in the plane.

Partial 2-trees are a subclass of planar graphs that includes series-parallel
graphs and outerplanar graphs as proper subclasses. In this paper, we show
that the complement of every partial 2-tree is in $\class{SEG}$.

\section{Definitions}
All the graphs that we consider shall be finite, simple and undirected. We
denote the vertex set of a graph $G$ by $V(G)$ and its edge set by $E(G)$.
Given a graph $G$ and $X\subseteq V(G)$, we denote the subgraph induced by
$V(G)\setminus X$ in $G$ as $G-X$. The complement of a graph $G$, denoted as
$\overline{G}$, is the graph with $V(\overline{G})=V(G)$ and
$E(\overline{G})=\{uv~|~u\not=v\mbox{ and }uv\not\in E(G)\}$.
\subsection{Partial 2-trees}
A \emph{2-tree} is defined as follows:
\begin{enumerate}
\item A single edge is a 2-tree.
\item If $G$ is a 2-tree, then the graph $G'$ with $V(G')=V(G)\cup\{v\}$
and $E(G')=E(G)\cup \{vx,vy\}$ where $xy\in E(G)$ is a 2-tree.
\end{enumerate}
A \emph{partial 2-tree} is any spanning subgraph of a 2-tree.
\subsection{Segments, rays and compatible segment representations}
Let $a,b\in \mathbb{R}^2$ be two points in the plane. Then:
\begin{itemize}
\item A \emph{segment} with end-points $a$ and $b$, denoted as $ab$, is the set
$\{a+\rho(b-a)~|~\rho\in [0,1]\}$. Any point on a segment that is not one of its
end-points is said to be an \emph{interior} point of the segment. For the
purposes of this paper, we shall assume that every segment has a non-zero
length---i.e., the end-points of a segment may not coincide.

\item A \emph{ray} starting at a point $a$ and passing through a point $b$ is
the set $\{a+\rho(b-a)~|~\rho\in [0,\infty)\}$. A ray has a single end-point,
which is its starting point.
\end{itemize}

Given $l_1$ and $l_2$, where each could be a segment or a ray, they are said to
\emph{cross} each other if $l_1\cap l_2$ consists exactly of a single point
that is not an end-point of either $l_1$ or $l_2$.
If $l_1\cap l_2=\emptyset$, then they are said to be \emph{disjoint}.

Given a segment graph $G$, there exists a function $f:V(G)\rightarrow
\mc{R}$ such that $\forall u,v\in V(G), u\not=v$, $f(u)\cap
f(v)\not=\emptyset\Leftrightarrow uv\in E(G)$ where \mc{R} is a collection of
segments. We say that \mc{R} is a ``segment representation'' of $G$.

\begin{definition}
Let $G$ be a partial 2-tree and let $G_T$ be a 2-tree of which $G$ is a
spanning subgraph. Let \mc{R} be a segment representation of $\overline{G}$ with
segments $\{s_u~|~u\in V(G)\}$. \mc{R} is said to be a segment representation of
$\overline{G}$ that is ``compatible with $G_T$'' if for each $uv\in E(G_T)$,
a ray $r_{uv}$ can be drawn in \mc{R} such that the collection of these rays
satisfies the following properties:
\begin{enumerate}
\item $r_{uv}$ starts from an interior point on one of $s_u$ or $s_v$ and passes
through an end-point of the other and meets no other points of either $s_u$ or
$s_v$.
\item $r_{uv}$ crosses every segment other than $s_u$ and $s_v$, and
\item $r_{uv}$ crosses every ray $r_{xy}$ where $xy\in E(G_T)\setminus
\{uv\}$.
\end{enumerate}
The rays $r_{uv}$ where $uv\in E(G_T)$ shall be called ``special rays''.
\end{definition}

\section{The construction}
\begin{theorem}
If $G$ is a partial 2-tree which is a spanning subgraph of a 2-tree $G_T$,
then $\overline{G}$ has a segment representation that is compatible with $G_T$.
\end{theorem}
\begin{proof}
We shall prove this by induction on $|V(G)|$. There is nothing to prove for
$|V(G)|<2$. If $|V(G)|=2$, then clearly, $\overline{G}$ has a segment
representation compatible with $G_T$ as shown in Figure \ref{figg0g1}.
\begin{figure}
 \center
 \input{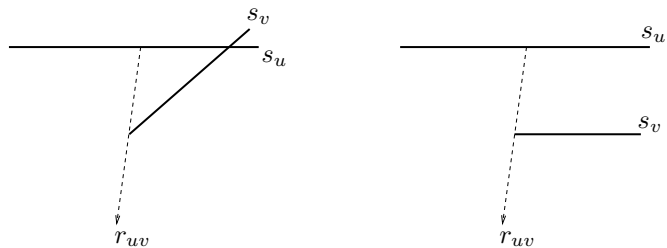}
 \caption{Segment representations that are compatible with $G_T$ for
$\overline{G}$ when $|V(G)|=2$.}
 \label{figg0g1}
\end{figure}
Consider a partial 2-tree $G$ with $|V(G)|>2$.
By definition of a 2-tree, there exists a degree 2 vertex $v$ in $G_T$ with
neighbours $x$ and $y$ such that $xy\in E(G_T)$ and $G_T-\{v\}$ is also a
2-tree. Let $G'$ be $G-\{v\}$. $G'$ is also a partial 2-tree as it is a spanning
subgraph of the 2-tree $G_T-\{v\}$. For ease of notation, we shall denote the
2-tree $G_T-\{v\}$ by $G'_T$. By our induction hypothesis, there is a segment
representation for $\overline{G'}$ that is compatible with $G'_T$. We shall show
how we can extend this to a segment representation for $\overline{G}$ that is
compatible with $G_T$, thereby completing the proof.

Let \mc{R'} be a segment representation of $\overline{G'}$ which is compatible
with $G'_T$ and let $s_u$ denote the segment corresponding to a vertex $u\in
V(\overline{G'})$. We shall add a new segment $s_v$ to \mc{R'} so that we get
the requried segment representation \mc{R} for $\overline{G}$.

Since \mc{R'} is compatible with $G'_T$, there is a special ray $r_{xy}$ in
\mc{R'} which we shall assume without loss of generality starts from an interior
point of $s_x$ and passes through an end-point of $s_y$. Let $p$ be the
starting point of the ray $r_{xy}$ on $s_x$ (refer Figure \ref{figconstr}).
\begin{figure}
\center
\input{constr.pstex_t}
\caption{Starting point of construction}
\label{figconstr}
\end{figure}
$l_1$ and $l_2$ are two parallel rays starting from points $p_1$ and $p_2$ on
$s_x$ on either side of $p$ and parallel to $r_{xy}$ such that they cross every
segment and special ray that $r_{xy}$ crosses. By our definition of crossing,
the ray $r_{xy}$ meets every segment and special ray that it crosses at a point
that is an end-point of neither of them. Therefore, we can always choose rays
$l_1$ and $l_2$ distinct from $r_{xy}$ as long as $s_x$ is not a point or
parallel to $r_{xy}$. Note that by our definition of
$r_{xy}$, $s_x$ cannot lie along $r_{xy}$ and $s_x$ can also not be a point
which ensures that the two rays $l_1$ and $l_2$ distinct from $r_{xy}$ can
be obtained. Of all the crossing points of segments and special rays on $l_1$,
let $q_1$ be the farthest from $p_1$. Let $q_2$ be the similarly defined point
on $l_2$.
Clearly, any segment or special ray that meets the segment $p_1p_2$ and
crosses $q_1q_2$ crosses every segment and special ray that $r_{xy}$ crosses.

\begin{figure}
 \begin{minipage}[b]{0.45\linewidth}
 \input{constr1.pstex_t}
 \caption{The case when $yv\in E(G)$}
 \label{figconstr1}
 \end{minipage}
 \hspace{0.4in}
 \begin{minipage}[b]{0.45\linewidth}
 \input{constr2.pstex_t}
 \vspace{0.15in}
 \caption{The case when $yv\not\in E(G)$}
 \label{figconstr2}
 \end{minipage}
\end{figure}

The segment $s_v$ is placed in the following way in $\mc{R'}$ to obtain
$\mc{R}$:

\medskip

\noindent{\itshape Case 1.} $yv\in E(G)$.
We place $s_v$ as shown in Figure \ref{figconstr1}. As $yv\not\in
E(\overline{G})$, the segment $s_v$ is disjoint from the segment $s_y$. Let us
first consider the case when $xv\not\in E(G)$.
In $\overline{G}$, $v$ is adjacent to all the vertices in $V(G)\setminus
\{v,y\}$. This requirement is satisfied, as since $s_v$ meets $p_1p_2$
and crosses $q_1q_2$, it crosses all the segments that cross $r_{xy}$ and meets
$s_x$ too. Moreover, $s_v$ also crosses all the special rays in \mc{R'}
including $r_{xy}$.
\com{ meaning that the special rays in \mc{R'} are all special rays in
\mc{R} as well.}
We will show that we can now draw two rays
\com{
But for \mc{R} to be a segment representation of $\overline{G}$ that is
compatible with $G_T$, there should exist two rays in \mc{R}, namely}
$r_{vy}$ and $r_{xv}$, so that they, together with the special rays in \mc{R'},
form the collection of special rays in \mc{R} which make it compatible with
$G_T$.
The rays $r_{vy}$ and $r_{xv}$ can be drawn as shown in the figure so that they
cross $s_x$ and $s_y$ respectively. In addition, both of them cross each other
and $r_{xy}$. Since they also meet $p_1p_2$ and cross $q_1q_2$, each of them
crosses all the special rays and segments that cross $r_{xy}$.
Thus $\mc{R}$ is a segment representation of $\overline{G}$ that is compatible
with $G_T$. Note that we
can draw the segment $s_v$ and the rays $r_{vy}$ and $r_{xv}$ in this way as
long as the segments $s_x$ and $s_y$ have non-zero length and $s_y$ does not lie
along the ray $r_{xy}$. But our definitions of segments and special rays ensure
that these pathological situations do not occur. Now, if $xv\in E(G)$, then
$s_v$ can be slightly shortened at the end $p_3$ so that it becomes
disjoint from $s_x$ without affecting any of the other arguments so that we
still obtain the segment representation \mc{R} for $\overline{G}$.

\medskip

\noindent{\itshape Case 2.} $yv\not\in E(G)$.
We place $s_v$, $r_{xv}$ and $r_{vy}$ as shown in Figure \ref{figconstr2}. The
rest of the argument is similar to that of Case 1.
\com{
In both cases, the new special rays $r_{xv}$ and $r_{vy}$ are shown. It easy to
check that these rays are indeed special rays. Now, if $xv\not\in E(G')$, we
can easily shorten the segment $s_v$ so that it does not touch $s_x$ without
affecting changing its adjacencies with any other segment. Thus, we now have a
segment representation of $\overline{G'}$ with $A'=A\cup\{(x,v),(v,y)\}$ as the
active set.
}

\medskip

We thus have a segment representation of $\overline{G}$ which is compatible with
$G_T$. This completes the proof.
\hfill\bbox
\end{proof}

\begin{corollary}
If $G$ is any partial 2-tree, then $\overline{G}$ is a segment graph.
\end{corollary}

\begin{corollary}
The complements of series-parallel graphs and outerplanar graphs are segment
graphs.
\end{corollary}
\begin{proof}
Series-parallel graphs and outerplanar graphs are subclasses of partial
2-trees.\hfill\qed
\end{proof}


\begin{thebibliography}{1}

\bibitem{Chalopin}
J\'{e}r\'{e}mie Chalopin and Daniel Gon\c{c}alves.
\newblock Every planar graph is the intersection graph of segments in the
  plane: extended abstract.
\newblock In {\em STOC '09: Proceedings of the 41st annual ACM symposium on
  Theory of computing}, pages 631--638, 2009.

\bibitem{KratoKube}
J.~Kratochv\'{i}l and Ale\v{s} Kub\v{e}na.
\newblock On intersection representations of co-planar graphs.
\newblock {\em Discrete Mathematics}, 178(1--3):251--255, 1998.

\bibitem{KratoNese}
J.~Kratochv\'{i}l and J.~Ne\v{s}et\v{r}il.
\newblock {INDEPENDENT SET} and {CLIQUE} problems in intersection-defined
  classes of graphs.
\newblock {\em Comment. Math. Univ. Carolinae}, 31(1):85--93, 1990.

\end{thebibliography}
\end{document}